\documentclass[11 pt]{amsart}
\usepackage{graphicx, amsmath, fullpage, amssymb, amsthm, amsfonts, color}
\usepackage{enumerate}
\newtheorem{theorem}{Theorem}[section]
\newtheorem{lemma}[theorem]{Lemma}

\theoremstyle{definition}
\newtheorem{definition}[theorem]{Definition}

\usepackage[toc,page]{appendix} 

\theoremstyle{remark}
\newtheorem{remark}[theorem]{Remark}
\numberwithin{equation}{section}
\usepackage[colorlinks,
            linkcolor=red,
            anchorcolor=red,
            citecolor=blue
            ]{hyperref}
 
\begin{document}

\title{Spectra of Random Regular Hypergraphs}
\author{Ioana Dumitriu}
\address{Department of Mathematics, University of California, San Diego, La Jolla, CA 92093}
\email{idumitriu@ucsd.edu}

\author{Yizhe Zhu}
\address{Department of Mathematics, University of California, San Diego, La Jolla, CA 92093}
\email{yiz084@ucsd.edu}
\thanks{}

%
\date{\today}

\begin{abstract}
 In this paper, we study the spectra of  regular hypergraphs following the definitions from Feng and Li (1996). 
Our main result is an analog of Alon's conjecture for the spectral gap of the random regular hypergraphs. We then relate the second eigenvalues to both its expansion  property and the mixing rate of the non-backtracking random walk on regular hypergraphs. We also prove the spectral gap for the non-backtracking operator of a random regular hypergraph introduced in Angelini et al. (2015). Finally, we obtain the convergence of  the empirical spectral distribution (ESD) for random regular hypergraphs in different regimes. Under certain conditions, we can show a local law for the ESD.
\end{abstract}

\maketitle

\section{Introduction}

Since their introduction in the early 1970s (see, for example, Berge's book
\cite{berge1973graphs}), hypergraphs have steadily risen to prominence,
both from a theoretical perspective and through their potential for
applications. Of the most recent fields to recognize their importance
we mention  machine learning, where they have been used to model data \cite{zhou2007learning}, including recommender systems \cite{tan2011using}, pattern recognition \cite{liu2011hypergraph} and bioinformatics \cite{tian2009hypergraph}.

As with graphs, one main feature for the study is graph expansion;
e.g., studies of regular graphs
\cite{alon1986eigenvalues,lubotzky1988ramanujan,friedman2003proof,alon2007non,bordenave2015new},
where all vertices have the same degree $d$,
and quasi-regular graphs (e.g., bipartite biregular
\cite{brito2016recovery,brito2018spectral}, where the graphs are
bipartite and the two classes are regular with degrees
$d_1$, respectively, $d_2$; or preference models and $k$-frames
\cite{wan2015class}, which generalizes these notions). The key property for graph expansion is fast random walk mixing. There are three main perspectives
on examining this property: vertex, edge, and spectral expansion
\cite{chung1997spectral}; the latter of these, the spectral gap, is
the most desirable feature as it controls the others (the bounds on
vertex and edge expansion generally involve the second eigenvalue of
the Laplacian of the graph).

For general, connected, simple graphs (possibly with loops), the Laplacian is a
scaled and shifted version of the adjacency matrix $A = (A_{ij})_{1
  \leq i,j \leq n}$, where $A_{ij} = 1$ if and only if $i$ and $j$ are connected by an edge and $0$ otherwise. The Laplacian
is defined by $L = I - D^{-1/2} A D^{-1/2}$, where $D$ is the diagonal matrix of
vertex degrees. 

As mentioned before, spectral expansion of a graph involves the spectral gap of
its Laplacian matrix; however, in the case of regular or bipartite
biregular graphs, looking at the adjacency matrix or the Laplacian
is equivalent (in the case of the regular ones, $D$ is a multiple of
the identity, and in the case of the bipartite biregular ones, the
block structure of the matrix ensures that $D^{-1/2}A D^{-1/2} =
\frac{1}{\sqrt{d_1 d_2}} A$). For regular and bipartite biregular graphs, the largest
(Perron-Frobenius) eigenvalue of the adjacency matrix is fixed
(it is $d$ for $d$-regular graphs and $\sqrt{d_1 d_2}$ for bipartite
biregular ones). So for these special cases, the study of the second
largest eigenvalue of the adjacency matrix is sufficient. As we show here, this will also be the
case for $(d,k)$-regular hypergraphs. 

The study of the spectral gap in $d$-regular graphs with fixed $d$ had
a first breakthrough in the Alon-Boppana bound
\cite{alon1986eigenvalues}, which states that the second largest
eigenvalue $\lambda:=\max\{\lambda_2,|\lambda_n|\}$ satisfies $\lambda\geq 2\sqrt{d-1}-o(1).$  Later, Friedman
\cite{friedman2003proof} proved Alon's conjecture
\cite{alon1986eigenvalues} that  a uniformly chosen random $d$-regular graphs have
$\lambda\leq 2\sqrt{d-1} +\epsilon$ for any $\epsilon>0$ with high
probability, as the number of vertices goes to infinity. Recently,
Bordenave \cite{bordenave2015new} gave a new proof that $\lambda_2\leq
2\sqrt{d-1}+\epsilon_n$ for a sequence $\epsilon_n\to 0$ as $n\to \infty$
based on the non-backtracking (Hashimoto) operator. Following the same idea in \cite{bordenave2015new}, Coste proved the spectral gap for  $d$-regular digraphs \cite{coste2017spectral} and Brito et
al. \cite{brito2018spectral} proved an analog of Alon's spectral gap
conjecture for random bipartite biregular graphs; for deterministic ones, the equivalent of
the Alon-Boppana bound had first been shown by Lin and Sol\'{e}
\cite{sole1996spectra}. 

It is thus fair to say that both graph expansion and the  spectral
gap in regular graphs and quasi-regular graphs are now very well understood; by contrast, despite the natural
applications and extension possibilities, hypergraph expansion is a
much less  understood area. The difficulty here is that it
is not immediately clear which operator or structure to associate to
the hypergraph. There are three main takes on this: the Feng-Li
approach  \cite{feng1996spectra}, which defined an adjacency matrix,
the Friedman-Wigderson tensor approach \cite{friedman1995second}, and the
Lu-Peng approach \cite{lu2011high,lu2012loose}, which defined a sequence of
Laplacian matrices through higher-order random walks. 

Several results on hypergraph expansion have been obtained using the
Friedman-Wigderson approach.  Hyperedge
expansion depending on the spectral norm of the associated tensor was studied in  the
original paper 
\cite{friedman1995second}, the relation between the spectral gap and quasirandom properties was discussed in Lenz and Mubayi \cite{lenz2015eigenvalues,lenz2017eigenvalues},  and an inverse expander mixing lemma was obtained in Cohen
et al. \cite{cohen2016inverse}. Very
recently, Li and Mohar \cite{li2018first} proved a generalization of the
Alon-Boppana bound to $(d,k)$-regular hypergraphs for their
adjacency tensors. On the other side, using the Feng-Li adjacency matrix
approach, the original paper \cite{feng1996spectra} proved the Alon-Boppana lower bound for the adjacency matrix of
regular hypergraphs, and then Li and Sol\'{e} \cite{sole1996spectra} defined a $(d,k)$-regular hypergraph to be \textit{Ramanujan} if any eigenvalue $\lambda\not=d(k-1)$  satisfies
\begin{align}\label{LISOLE}
|\lambda- k+2|\leq 2\sqrt{(d-1)(k-1)}.
\end{align} Ramanujan hypergraphs were further studied in \cite{martinez2001some,li2004ramanujan,sarveniazi2007explicit}.
 Note that when $k=2$ (when the hypergraphs are actual graphs), this definition coincides with the definition for Ramanujan graphs. The adjacency matrices and Laplacian matrices of general uniform hypergraphs were analyzed in \cite{banerjee2017spectrum}, where the relation between eigenvalues  and diameters, random walks, Ricci curvature of the hypergraphs were studied.
	
In this paper, we fill in the gaps in the literature by showing a
spectral gap for the adjacency matrix of a hypergraph, following the
Feng-Li definition; we connect it to the mixing rate of the hypergraph 
random walk considered in \cite{zhou2007learning} and subsequently
studied in \cite{cooper2013cover,helali2019hitting}, and we also show that this gap governs
hyperedge and vertex expansion of the hypergraph, thus completing the
parallel with graph results. Specifically, for $(d,k)$-regular
hypergraphs and their adjacency matrices (the precise definitions are given in the next section), we prove the following:		
\begin{itemize}
\item Hyperedge and vertex expansion are controlled by the second
  eigenvalue of the adjacency matrix.
\item The mixing rate of the random walk is controlled by the second
  eigenvalue of the adjacency matrix.
\item The uniformly random $(d,k)$-regular hypergraph model has a spectral gap. This is by far the most exciting result, and it
  turns out to be a simple consequence of the spectral gap of
  uniformly random bipartite biregular graphs 
  \cite{brito2018spectral}. Our result shows that, asymptotically,
  almost all $(d,k)$-regular hypergraphs are almost Ramanujan in the sense of
  Li-Sol\'{e} (see \eqref{LISOLE}).
\end{itemize}
 	
Other results include the spectral gap and description for the
spectrum of the non-backtracking operator of the hypergraph, the
limiting empirical distribution for the spectrum of the adjacency
matrix of the uniformly random $(d,k)$-regular hypergraph in different
regimes (which was studied by Feng and Li in \cite{feng1996spectra}
for deterministic sequences of hypergraphs with few cycles and fixed
$d,k$), and a sort of local law of this empirical spectral
distribution.  
 	
Our main methodology is to translate the results from bipartite biregular graphs by using the bijection between the spectra (Lemma \ref{bijection}). While this bijection has been known for a long time, the results on bipartite biregular graphs \cite{dumitriu2016marvcenko,brito2018spectral} (especially the spectral gap)  are quite recent.

Our spectral gap results are linked to the random walk and offer
better control over the mixing rate. Together with the Alon-Boppana
result established by Feng-Li \cite{feng1996spectra}, they give
complete control over the behavior of the random walk and hyperedge/vertex
expansion. In our view, this establishes the adjacency matrix
perspective of Feng and Li as ultimately more useful not just
theoretically, but possibly computationally as well, since computing
second eigenvalues of matrices is achievable in polynomial time, whereas the complexity of computing spectral norms of tensors is NP-hard \cite{hillar2013most}. 

The rest of the paper is structured as follows. In Section
\ref{sec:prelim} we provide definitions and properties of hypergraphs
that we use in the paper. In Section \ref{sec:deterministic} we show
that several expansion properties of $(d,k)$-regular hypergraphs are related
to the second eigenvalues of their adjacency matrices. In Section \ref{sec:spectralgap} we prove
the analog of Friedman's second eigenvalue theorem for uniformly random $(d,k)$-regular
hypergraphs. The spectra of the non-backtracking operator for the
hypergraph are analyzed in Section \ref{sec:NBO}. Finally,
we study the empirical spectral distributions of uniformly random $(d,k)$-regular hypergraphs in Section \ref{sec:ESD}.

\section{Preliminaries}\label{sec:prelim}

\begin{definition}[hypergraph]
	A \textit{hypergraph} $H$ consists of a set $V$ of vertices and a set $E$ of hyperedges such that each hyperedge is a nonempty subset of $V$.  A hypergraph $H$ is \textit{$k$-uniform} for an integer $k\geq 2$ if every hyperedge $e\in E$ contains exactly $k$ vertices. The \textit{degree} of $i$, denoted   $\deg(i)$, is the number of all hyperedges incident to $i$.
 A hypergraph is \textit{$d$-regular} if all of its vertices have degree $d$.
A hypergraph is \textit{$(d,k)$-regular} if it is both $d$-regular and $k$-uniform.  A vertex $i$ is \textit{incident} to a hyperedge $e$ if and only if $v$ is an element of $e$. We can define the \textit{incidence matrix} $X$ of a hypergraph to be a $|V|\times |E|$ matrix indexed by elements in $V$ and $E$ such that $X_{i,e}=1$ if $i\in e$ and $0$ otherwise. Moreover, if we regard $X$ as the adjacency matrix of a graph, it defines a bipartite graph $G$ with two vertex sets being $V$ and $E$. We call $G$  the  \textit{bipartite graph associated to $H$}. 
\end{definition}
\begin{definition}[walks and cycles]\label{cycledefinition}
A \textit{walk} of length $l$ on a hypergraph $H$ is a sequence \[(v_0,e_1,v_1,\cdots ,e_l, v_l)\] such that $v_{j-1}\not=v_{j}$ and $\{v_{j-1},v_j\}\subset e_j$ for all $1\leq j\leq l$. A walk is closed if $v_0=v_l$. A cycle of length $l$ in a hypergraph $H$ is a closed walk $(v_0,e_1,\dots, v_{l-1}, e_l,v_{l+1})$ such that 
\begin{itemize}
	\item $|\{e_1,\dots, e_{l}\}|=l$ (all hyperedges are distinct);
	\item $|\{v_0,\dots v_{l-1}\}|=l$, $v_{l+1}=v_0$ (all vertices are distinct subject to $v_{l+1}=v_0$).
\end{itemize}
In the associated bipartite graph $G$, a cycle of length $l$ in $H$ corresponds to a cycle of length $2l$. We say $H$ is connected if for any $i,j\in V$, there is a walk between $i,j$. It's easy to see $H$ is connected if and only if the corresponding bipartite graph $G$ is connected.
\end{definition}

\begin{definition}[adjacency matrix]
For  a hypergraph $H$ with $n$ vertices, we associate a $n\times n$ symmetric  matrix $A$ called the \textit{adjacency matrix} of $H$. For $i\not=j$, $A_{ij}$ is the number of hyperedges containing both $i$ and $j$ and $A_{ii}=0$ for all $1\leq i\leq n$.
\end{definition}
 If $H$ is $2$-uniform, this is the adjacency matrix of an ordinary graph.  
The largest eigenvalue of $A$ for $(d,k)$-regular hypergraphs is $d(k-1)$ with eigenvector $\frac{1}{\sqrt{n}}(1,\dots,1)$.

\section{Expansion and mixing properties of regular hypergraphs}\label{sec:deterministic}
In this section, we  relate the expansion property of a regular hypergraph to its second eigenvalue. We prove  results on expander mixing and vertex expansion, and compute the mixing rate of simple random walks and non-backtracking random walks. These results follow easily from the same methodology used in Chung's book \cite{chung1997spectral}. Let $H=(V,E)$ be a $(d,k)$-regular hypergraph, for any subsets $V_1, V_2\subset V$, define 
\begin{align*}
e(V_1,V_2):=| \{(i,j,e): i\in V_1, j\in V_2, \{i,j\}\in e\subset E\}	|
\end{align*}
which counts the number of hyperedges between vertex set $V_1, V_2$ with multiplicity. For each hyperedge $e$, the multiplicity is given by $|e\cap V_1|\cdot |e\cap V_2|$. We first provide an edge mixing result whose equivalence for  regular graphs is given in \cite{chung1997spectral}.
\begin{theorem}[expander mixing]\label{expander}
	Let $H=(V,E)$ be a $(d,k)$-regular hypergraph with  adjacency matrix $A$. Let $\lambda=\max \{\lambda_2(A), |\lambda_n(A)|\}$. The following holds: for any subsets $V_1, V_2\subset V$,
	\begin{align*}
		\left|e(V_1,V_2)-	\frac{d(k-1)}{n}|V_1|\cdot |V_2|\right| \leq \lambda \sqrt{|V_1|\cdot |V_2|\left(1-\frac{|V_1|}{n}\right) \left(1-\frac{|V_2|}{n}\right)}.
	\end{align*}
 
\end{theorem}

\begin{remark}
The above result is qualitatively different from the expander mixing lemma for $k$-uniform regular graphs studied in \cite{friedman1995second,  cohen2016inverse}.  Their result considers the number of hyperedges between any $k$ subsets of $V$ and  the parameter $\lambda$ there is  the spectral norm of a tensor associated with the hypergraph.
\end{remark}

\begin{proof}
Let $1_{V_i}$ be the indicator vector of the set $V_i$ for $i=1,2$. Let $v_1,\dots, v_n$ be the unit eigenvector associated to $\lambda_1,\dots \lambda_n$ of $A$. We have the following decomposition of $1_{V_1}, 1_{V_2}$:
\begin{align*}
1_{V_1}=\sum_{i=1}^n \alpha_i v_i, \quad 
	1_{V_2}=\sum_{i=1}^n \beta_i v_i
\end{align*}
for some numbers $\alpha_i, \beta_i, 1\leq i\leq n$. Recall that $\lambda_1=d(k-1)$ and $v_1=\frac{1}{\sqrt{n}}(1,\dots 1)^{\top}$. We have $\alpha_1=\langle 1_{V_1}, v_1\rangle=\frac{|V_1|}{\sqrt n},  \beta_1=\langle 1_{V_2}, v_1\rangle=\frac{|V_2|}{\sqrt n}.$
From the definition of $e(V_1,V_2)$, 
\begin{align*}
e(V_1,V_2)&=\sum_{1\leq i,j\leq n}1_{V_1}(i)1_{V_2}(j)A_{ij}
=	\langle 1_{V_1}, A1_{V_2}\rangle\\
&=\lambda_1 \alpha_1\beta_1+\sum_{i\geq 2}\lambda_i\alpha_i\beta_i
=\frac{d(k-1)}{n}|V_1|\cdot |V_2|+\sum_{i\geq 2}\lambda_i\alpha_i\beta_i.
\end{align*}
Therefore by the Cauchy-Schwarz inequality,
\begin{align*}
\left|e(V_1,V_2)-	\frac{d(k-1)}{n}|V_1|\cdot |V_2|\right| \leq \lambda \sum_{i\geq 2}|\alpha_i\beta_i|\leq \lambda \sqrt{\sum_{i\geq 2}\alpha_i^2}\sqrt{\sum_{i\geq 2}\beta_i^2}.
\end{align*}
Note that $$ \sum_{i\geq 2}\alpha_i^2= \sum_{i\geq 1}\alpha_i^2-\frac{|V_1|^2}{n}=|V_1|\left(1-\frac{|V_1|}{n}\right),  $$ and similarly,
$\sum_{i\geq 2}\beta_i^2= |V_2|\left(1-\frac{|V_2|}{n}\right).
$ This implies 
\begin{align*}
\left|e(V_1,V_2)-	\frac{d(k-1)}{n}|V_1|\cdot |V_2|\right| \leq \lambda \sum_{i\geq 2}|\alpha_i\beta_i|\leq \lambda \sqrt{|V_1|\cdot |V_2|\left(1-\frac{|V_1|}{n}\right) \left(1-\frac{|V_2|}{n}\right)}.
\end{align*}
\end{proof}

For any subset $S\subset V$, we define its \textit{neighborhood set} to be $$N(S):=\{i: \text{there exists } j\in S \text{ such that } \{i,j\}\subset e \text{ for some $e\in E$}\}.$$
 We have the following result on vertex expansion of regular hypergraphs.
 \begin{theorem}[vertex expansion]\label{vertexexpansion}
 	Let $H=(V,E)$ be a $(d,k)$-regular hypergraph with  adjacency matrix $A$. Let $\lambda=\max \{\lambda_2(A), |\lambda_n(A)|\}$.  For any subset $S\subset V$, we have that 
 	\begin{align}\label{eq:NSS}
 	\frac{|N(S)|}{|S|}\geq \frac{1}{1-\left(1-\frac{\lambda^2}{d^2(k-1)^2}\right)\left(1-\frac{|S|}{n}\right)}.	
 	\end{align}
 \end{theorem}
 
\begin{proof}
 Let $1_S$ be the indicator vector of the set $S$ with the decomposition $1_S=\sum_{i=1}^n\gamma_i v_i$, where $\gamma_i, 1\leq i\leq n$ are constants and  $v_i, 1\leq i\leq n$ are the unit eigenvectors of $A$ associated to $\lambda_1,\dots,\lambda_n$, respectively. Then we know $\gamma_1=\frac{|S|}{\sqrt{n}}$ and 
 \begin{align}
 \|A1_S\|_2^2&=\sum_{i=1}^n \lambda_i^2\gamma_i^2\leq  d^2(k-1)^2\frac{|S|^2}{n} +\lambda^2(\sum_{i\geq 2}\gamma_i^2) \notag\\
 &= d^2(k-1)^2\frac{|S|^2}{n} +\lambda^2\left(|S|-\frac{|S|^2}{n^2}\right)=(d^2(k-1)^2-\lambda^2)\frac{|S|^2}{n}+\lambda^2|S|. \label{eq:A1s}
 \end{align}
On the other hand,
\begin{align}
\|A1_S\|_2^2&=\langle A1_S, A1_S\rangle	
=\sum_{i=1}^n(\sum_{k\in S}A_{ik})^2
=\sum_{i=1}^n e(\{i\},S)^2=\sum_{i\in N(S)} e(\{i\},S)^2,
\end{align}
and by Cauchy-Schwartz inequality,
\begin{align}
\sum_{i\in N(S)} e(\{i\},S)^2\geq \frac{\left(\sum_{i\in N(S)}e(S,\{i\})\right)^2}{|N(S)|}.	
\end{align}
The quantity 
\begin{align}
\sum_{i\in N(S)}e(S,\{i\})=e(S,N(S))=|\{(i,j,e): i\in S, j\in N(S),  \{i,j\}\subset e\in E\}|
\end{align}
counts the number of hyperedges between $S$ and $N(S)$ with multiplicity. We then have \begin{align}\label{eq:ESN}
    e(S,N(S))=|S|(k-1)d.
\end{align} 
Putting Equations \eqref{eq:A1s}-\eqref{eq:ESN} together, we obtain \eqref{eq:NSS}.
 \end{proof}

For the rest of this section, we compute the mixing rates of random walks on hypergraphs. The  simple random walk on a general hypergraph was first defined in \cite{zhou2007learning}, where the authors gave a random walk explanation of the spectral methods for clustering and segmentation on hypergraphs, which generalized the result in Meila and Shi \cite{maila2001random} for graphs.  A quantum version of random walks on regular hypergraphs was recently studied by Liu et al. \cite{liu2018quantum}. 

The simple random walk on $k$-uniform hypergraphs has the following transition rule. Start at a vertex $v_0$. If at the $t$-th step we are at  vertex $v_t$, we first choose a hyperedge $e$ uniformly over all hyperedges incident with $v_t$, and then choose a vertex  $v_{t+1}\in e, v_{t+1}\not=v_{t}$ uniformly at random. 	 The sequence of random vertices $(v_t,t\geq 0)$ is a Markov chain. 
 It generalizes the simple random walk on graphs. 
  We denote by $P=(P_{ij})_{1\leq i,j\leq n}$  the \textit{transition matrix} for the Markov chain and  let $D$ be the diagonal matrix with $D_{ii}=\deg(i), 1\leq i\leq n$.  From the definition of the simple random walk on hypergraphs,
for any $(d,k)$-regular hypergraphs with adjacency matrix $A$,
the transition matrix satisfies $
		P=\frac{1}{d(k-1)}A.
	$

It's known (see for example \cite{lovasz1994random}) that for any graph (or multigraph) $G$, if $G$ is connected and non-bipartite, then it has a unique stationary distribution. For $d$-regular graphs, being connected and non-bipartite is equivalent to requiring $\lambda=\max\{\lambda_2(A),|\lambda_n(A)|\}<d,$
see for example \cite{alon2007non}.  The simple random walk on $(d,k)$-regular hypergraphs $H=(V,E)$ can also be seen as a simple random walk on a multigraph $G_H$ on $V$, where the number of edges between $i,j$ in $G_H$ is $A_{ij}$. The adjacency matrix  of $G_H$ is the same as the adjacency matrix of $H$. Therefore the simple random walk on $H$ converges to a unique stationary distribution if and only if the multigraph $G_H$ is connected and non-bipartite. These two conditions can be satisfied as long as we have the following condition on the second eigenvalue.
\begin{lemma}
Let $H$ be a $(d,k)$-regular hypergraph with adjacency matrix $A$. The simple random walk on $H$ converges to a stationary distribution  if and only if $$\lambda=\max\{\lambda_2(A),|\lambda_n(A)|\}<d(k-1).$$
\end{lemma}
\begin{proof}
If the corresponding multigraph $G_H$ is bipartite, then we have $\lambda_n=-\lambda_1=-d(k-1)$. If $G_H$ is not connected, then it has at least two connected components, the largest eigenvalue will have multiplicity $\geq 2$, which implies  $\lambda=d(k-1)$. Therefore   $\lambda<d(k-1)$ if and only if $G_H$ is non-bipartite and connected. From the general theory of Markov chains on graphs and multigraphs, the simple random walk on $G_H$ converges to a stationary distribution. Therefore the simple random walk on $H$ converges to a stationary distribution.
\end{proof}

For any $(d,k)$-regular hypergraph $H$ with $\lambda<d(k-1)$, a simple calculation shows that the stationary distribution is $\pi(i)=\frac{1}{n}$ for all $i\in V$.
The mixing rate of the simple random walk on hypergraphs, which measures how fast the Markov chain converges to the stationary distribution,  is defined by
$$
\rho(H):=\limsup_{l\to\infty}\max_{i,j\in V}|(P^{l})_{ij}-\pi(j)|^{1/l},	
$$
where $\pi$ is the unique stationary distribution on $V$. 
Let $\lambda_1\geq\lambda_2\geq \cdots \geq \lambda_n$ be the eigenvalues of $A$ and we define the second eigenvalue in absolute value of $A$ by $\lambda:=\max\{\lambda_2,|\lambda_n|\}$.

The non-backtracking walk on a hypergraph is defined in \cite{storm2006zeta}, as a generalization of non-backtracking walk on a graph. Recall a walk of length $l$ in a  hypergraph is a sequence \[w=(v_0,e_1,v_1,e_2,\dots v_{l-1},e_l,v_l)\] such that $v_i\not=v_{i+1}$ and $\{v_i,v_{i+1}\}\subset e_{i+1}$ for all $0\leq i\leq l-1$.  We say $w$ is a \textit{non-backtracking walk} if $e_i\not=e_{i+1}$ for $1\leq i\leq l-1$. Define a \textit{non-backtracking random walk} of length $l$ on $H$ from some given vertex $v_0\in V$, to be a uniformly chosen member of the set of non-backtracking walks of length $l$ starting at $v_0$.
Let
\begin{align}\vec{E}(H):=\{(i,e): i\in V(H), e\in E(H), i\in e\}\end{align} be the set of oriented hyperedges of a $k$-uniform hypergraph $H$. Similar to case for regular graphs in \cite{alon2007non},
we can also consider the non-backtracking random walk on $H$ starting from an initial vertex $v_0$ as a Markov chain $\{X_t\}_{t\geq 0}$ with a state space $\vec{E}(H)$ in the following way. The distribution of the initial state is given by
$$\mathbb P(X_0=(v_0,e))=\frac{\mathbf{1}_{v_0\in e}}{\deg(v_0)}, 
$$ for any $e\in E(H)$. The transition probability is given by
\begin{align*}
&\mathbb P(X_{t+1}=(u,f)\mid X_t=(v,e)) \\
=&\begin{cases}
\frac{1}{(k-1)(\deg (u)-1)} & \text{if } u\in e, u\not=v\in V(H), \text{ and } f\not=e \in E(H),\\
0 &	\text{otherwise}.
\end{cases}
\end{align*}

 Notice that if $H$ is a $(d,k)$-regular hypergraph with $(d,k)=(2,2)$, then $H$ is a 2-regular graph, which is a disjoint union of cycles. The  non-backtracking random walk on $H$ is periodic and does not converge to a stationary
distribution. Given a $(d,k)$-regular hypergraph $H=(V,E)$ with $(d,k)\not=(2,2)$, let $\tilde{P}_{i,j}^{(l)}$ be the transition probability that a non-backtracking random walk of length $l$ on $H$ starts at  $i$ and ends at $j$.
	Define 
$$
\tilde{\rho}(H):=\limsup_{l\to\infty}\max_{i,j\in V}\left|\tilde{P}^{(l)}_{ij}-\frac{1}{n}\right|^{1/l}	
$$
to be the mixing rate of the non-backtracking random walk.
As a generalization of the result in \cite{alon2007non}, we can connect the second eigenvalue  of regular hypergraphs to the mixing rate of non-backtracking random walk. It turns out that  similar results were already studied in \cite{cioaba2015mixing} for clique-wise non-backtracking walks on regular graphs. Especially,  Corollary 1.3 in \cite{cioaba2015mixing} is equivalent to the following theorem.  We include a proof here for completeness.
 
 \begin{theorem}\label{mixingrate}
Let $H$ be a  $(d,k)$-regular hypergraph with $ d,k\geq 2$  whose adjacency matrix has the second largest eigenvalue in absolute value $\lambda:=\max\{\lambda_2,|\lambda_n|\}<d(k-1)$, then 
\begin{enumerate}
	\item the mixing rate of the simple random walk on $H$ is  $\rho(H)=\frac{\lambda}{(k-1)d}$.
\item Assume further that $(d,k)\not=(2,2)$. Define a function $\psi: [0,\infty)\to \mathbb R$ as 
	\begin{align*}
	\psi(x):=\begin{cases}
		x+\sqrt{x^2-1}& \text{if } x\geq 1,\\
		1& \text{if } 0\leq x\leq 1.
	\end{cases}
	\end{align*}
	Then a non-backtracking random walk on $H$ converges to the uniform distribution, and its mixing rate $\tilde{\rho}(H)$ satisfies 
	$\tilde{\rho}(H)=\frac{1}{\sqrt{(d-1)(k-1)}}	~\psi \left(\frac{\lambda}{2\sqrt{(k-1)(d-1)}}\right).
	$
\end{enumerate}
\end{theorem}

\begin{proof}
(1) We first consider simple random walks.
For any $l\geq 1$, $ P^{l}=\frac{1}{((k-1)d)^l}A^l$ and the vector $v_1=\frac{1}{\sqrt{n}}(1,\dots, 1)$ is an eigenvector of $P^{l}$ corresponding to the unique largest eigenvalue  $1$. Let
		$\mu(l)=\max\{|\lambda_2(P^{l})|,|\lambda_n(P^{l})|\},$
		 we have
\[
		\max_{ij}|P_{ij}^{l}-\frac{1}{n}|=\max_{ij}|\langle (P^{l}-v_1^{\top}v_1)e_i,e_j \rangle |\leq \max_{|u|=|v|=1}|\langle (P^{l}-v_1^{\top}v_1)u,v\rangle|=\mu(l)=\frac{\lambda^l}{(k-1)^ld^l}.
		\]
		This implies  $\rho(H)\leq \frac{\lambda}{(k-1)d}$. 
		On the other hand, let $J$ be a $n\times n$ matrix whose entries are all $1$, we have
\begin{align*}
		\max_{ij}|P_{ij}^{l}-\frac{1}{n}|&\geq \frac{1}{n}\sqrt{\sum_{ij}\left|(P^l)_{ij}-\frac{1}{n}\right|^2}=\frac{1}{n}\left\|P^{l}-\frac{1}{n}J \right\|_F=\frac{1}{n}\sqrt{\sum_{2\leq i\leq n}\lambda_i^2(P^l)}\\
		&\geq \frac{\mu(l)}{n}=\frac{1}{n}\frac{\lambda^l}{(k-1)^ld^l},
\end{align*}
		which implies $ \rho(H)\geq \frac{\lambda}{(k-1)d}$. This completes the proof of part (1) of Theorem \ref{mixingrate}.
		
		 For part (2), we follow the steps in \cite{alon2007non}. Recall that the Chebyshev polynomials satisfy the following recurrence relation:
$
	U_{k+1}(x)=2xU_k(x)-U_{k-1}(x), \forall k\geq 0.
$
We also define $U_{-1}(x)=0, U_0(x)=1$.  Let $A$ be the adjacency matrix of $H$ and define the matrix $A^{(l)}$ such that $A_{ij}^{(l)}$ is the number of non-backtracking walks of length $l$ from $i$ to $j$ for all $i,j$. By definition, the matrices $A^{(l)}$ satisfy the following recurrence:
\begin{align}
\begin{cases}
	A^{(1)}=A, \quad A^{(2)}=A^2-(k-1)dI,\\
	A^{(l+1)}=AA^{(l)}-(k-1)(d-1)A^{(l-1)} \text{ for $l\geq 2$,}\label{eliminate}
\end{cases} 
\end{align}
where  $(k-1)dI$ in the first equation eliminates the diagonal of $A^2$ to avoid backtracking, and $(k-1)(d-1)A^{(l-1)}$ in the second equation of \eqref{eliminate} eliminates the walk which backtracks  in the $(l+1)$-st step.  We claim that
\begin{align}\label{claim}
A^{(l)}&=\sqrt{(k-1)^l d(d-1)^{l-1}}Q_l\left(\frac{A}{2\sqrt{(k-1)(d-1)}}\right),	
\end{align}
where 
$Q_l(x)=\sqrt{\frac{d-1}{d}}U_l(x)-\frac{1}{\sqrt{d(d-1)}}U_{l-2}(x)$
 for all $l\geq 1$.
To see this, let $$f(A,l):=\sqrt{(k-1)^l d(d-1)^{l-1}}Q_l\left(\frac{A}{2\sqrt{(k-1)(d-1)}}\right).$$
Since $U_1(x)=2x, U_2(x)=4x^2-1$, we have 
$$Q_1(x)=2\sqrt{\frac{d-1}{d}}x, \quad Q_2(x)=\sqrt{\frac{d-1}{d}}(4x^2-1)-\frac{1}{\sqrt{d(d-1)}}.
$$
We can check that
\begin{align*}f(A,1)&=\sqrt{(k-1)d} \cdot Q_1\left(\frac{A}{2\sqrt{(k-1)(d-1)}}\right)=A=A^{(1)},\\
	f(A,2)&=\sqrt{(k-1)^2d(d-1)}\cdot Q_2\left(\frac{A}{2\sqrt{(k-1)(d-1)}}\right)=A^2-(k-1)dI=A^{(2)}.
\end{align*}
Therefore \eqref{claim} holds for $l=1,2$. Since $Q_l(x)$ is a linear combination of $U_{l-2},U_l$,  it satisfies the recurrence
$Q_{k+1}(x)=2xQ_k(x)-Q_{k-1}(x).
$ Therefore by induction we have $f(A,l)=A^{(l)}$ for all $l\geq 1$.   Recall $\tilde{P}_{i,j}^{(l)}$ is the probability that a non-backtracking random walk of length $l$ on $H$  starts from $i$ and ends in $j$. The number of all possible non-backtracking walks of length $l$ starting from $i$ is $d(k-1)((k-1)(d-1))^{l-1}.$
This is because for the first step we have $d(k-1)$ many choices for hyperedges and vertices, and for the remaining $(l-1)$ steps we have $((k-1)(d-1))^{l-1}$ many choices in total.
Normalizing $A^{(l)}$ yields
\begin{align}\label{transition}
	\tilde{P}_{ij}^{(l)}=\frac{A^{(l)}_{ij}}{(k-1)d((k-1)(d-1))^{l-1}}=\frac{A^{(l)}_{ij}}{d(d-1)^{l-1}(k-1)^l}.
\end{align}
Let $\tilde{\mu}_1(l)=1, \tilde{\mu}_2(l)\geq \cdots\geq  \tilde{\mu}_n(l)$ be the eigenvalues of $\tilde{P}^{(l)}$,   $\tilde{\mu}(l):=\max\{|\mu_2(l)|, |\mu_n(l)|\}.$
We obtain that $\tilde{P}^{(l)}$ is precisely the transition matrix of a non-backtracking random walk of length $l$. Same as Claim 2.2 in \cite{alon2007non}, we have 
\begin{align}\label{LUB}
\frac{\tilde{\mu}(l)}{n}\leq \max_{ij} \left|\tilde{P}_{ij}^{(l)}-\frac{1}{n}\right|\leq \tilde{\mu}(l).
\end{align}
We sketch the proof of \eqref{LUB} here. Since $\tilde{P}^{(l)}$ is doubly stochastic, $v_1=\frac{1}{\sqrt{n}}(1,\dots, 1)$ is an eigenvector of $\tilde{P}^{(l)}$ corresponding to the largest eigenvalue  $1$. We have
		$$
		\max_{ij}|\tilde{P}_{ij}^{(l)}-\frac{1}{n}|=\max_{ij}|\langle (\tilde{P}^{(l)}-v_1^{\top}v_1)e_i,e_j\rangle|\leq \max_{\|u\|_2=\|v\|_2=1}|\langle (\tilde{P}^{(l)}-v_1^{\top}v_1)u,v \rangle|=\tilde{\mu}(l).
		$$
On the other hand, let $J$ be as above, we have
\[
		\max_{ij}|\tilde{P}^{(l)}_{ij}-\frac{1}{n}|\geq \frac{1}{n}\sqrt{\sum_{ij}\left|\tilde{P}^{(l)}_{ij}-\frac{1}{n}\right|^2}=\frac{1}{n}\|\tilde{P}^{(l)}-\frac{1}{n}J\|_F=\frac{1}{n}\sqrt{\sum_{2\leq i\leq n}\lambda_i^2(\tilde{P}^{(l)})}\geq \frac{\tilde{\mu}(l)}{n}.
\]
	Therefore 
$$
	\tilde{\rho}(H)=\limsup_{l\to\infty}\max_{i,j\in V}\left|\tilde{P}^{(l)}_{ij}-\frac{1}{n}\right|^{1/l}=\limsup_{l\to\infty}\tilde{\mu}(l)^{1/l}.
$$
By \eqref{claim} and \eqref{transition}, for $1\leq i\leq n$,
\[
	\tilde{\mu}_i(l)=\frac{\lambda_i(A^{(l)})}{d(d-1)^{l-1}(k-1)^l}=\frac{1}{\sqrt{d(d-1)^{l-1}(k-1)^l}}Q_l\left(\frac{\lambda_i(A)}{2\sqrt{(k-1)(d-1)}}\right).
\]
From Lemma 2.3. in \cite{alon2007non},  
\[
	\limsup_{l\to\infty}|Q_l(x)|^{1/l}=\psi(|x|)=\begin{cases}
		1 & 0\leq x\leq 1,\\
		|x|+\sqrt{x^2-1} & x\in \mathbb R\setminus[-1,1].
	\end{cases}
\]
Therefore
$
\tilde{\rho}(H)=\frac{1}{\sqrt{(d-1)(k-1)}}	~\psi \left(\frac{\lambda}{2\sqrt{(k-1)(d-1)}}\right).
$ This completes the proof.
\end{proof}

\section{Spectral gap of random regular hypergraphs}\label{sec:spectralgap}

Let $\mathcal G(n,m,d_1,d_2)$ be the set of all  simple bipartite biregular random graphs  with vertex set $V=V_1\cup V_2$ such that  $|V_1|=n, |V_2|=m$, and every vertex in $V_i$ has degree $d_i$ for $i=1,2$.  Here we must have $nd_1=md_2=|E|$.  Let $\mathcal H (n,d,k)$ be the set of all simple (without multiple hyperedges) $(d,k)$-regular hypergraphs with labelled  {vertex set $[n]$ and $\frac{nd}{k}$ many labelled hyperedges denoted by $\{e_1,\dots,e_{nd/k}\}$.} 

\begin{remark}\label{dualremark}
	 From this section on, we always assume $d\geq k$ for simplicity, since a $(d,k)$-regular hypergraph, its dual hypergraph is $(k,d)$-regular, and they have the same associated bipartite biregular graph by swapping the vertex sets $V_1$ and $V_2$. 
\end{remark}

It's well known (see for example \cite{feng1996spectra}) that there exists a bijection between regular multi-hypergraphs and bipartite biregular graphs.  See  Figure \ref{fig:bijection} as an example of the bijection. For a given bipartite biregular graph, if there are two vertices in $V_2$ that share the same set of neighbors in $V_1$, the corresponding regular hypergraph will have  multiple hyperedges, see Figure \ref{fig:forbidden}.  Let $\mathcal G'(n,m,d_1,d_2)$ be a subset of $\mathcal G(n,m,d_1,d_2)$ such that for any $G\in  \mathcal G'(n,m,d_1,d_2)$, the vertices in $V_2$ have different sets of neighborhoods in $V_1$.  We obtain the following bijection.

 \begin{figure}[ht]
\centering
\includegraphics[width=0.5\linewidth]{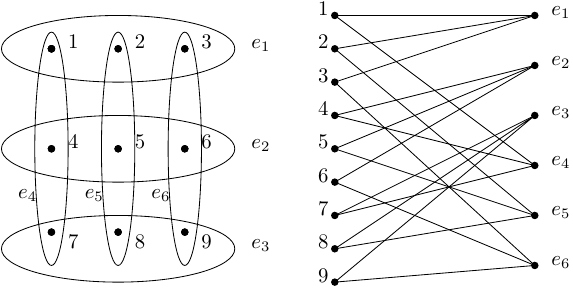}
\caption{a $(2,3)$-regular hypergraph and its associated bipartite biregular graph where all vertices in $V_2$ have different neighborhoods in $V_1$}
\label{fig:bijection}
\end{figure}

\begin{figure}[ht]
\centering
	\includegraphics[width=0.2\linewidth]{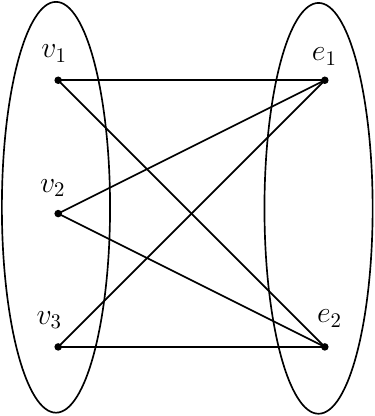}
	\caption{a subgraph in a  bipartite biregular graph which gives  multiple hyperedges $e_1$ and $e_2$ in the corresponding regular hypergraph}\label{fig:forbidden}
\end{figure}


\begin{lemma}\label{bijection}
	There is a bijection between the set $\mathcal H(n,d,k)$  and the set  $\mathcal G'\left(n, nd/k,d,k\right)$.
\end{lemma}

\begin{proof}
	Let $G\in \mathcal G'\left(n, nd/k,d,k\right)$ be an $(n, nd/k,d,k)$-bipartite biregular graph, and $A_G$ be its adjacency matrix, we then have 
	$A_G=\begin{pmatrix}
	0 & X\\
	X^{\top}& 0
\end{pmatrix},$ where $X$ is a $n \times (nd/k)$ matrix with entries $X_{ij}=1$ if and only there is an edge between $i\in V_1, j\in V_2$. We can then construct a regular hypergraph $H=(V(H),E(H))$ from $X$ with $V(H)=V_1$. There exists an edge $e_j=\{i_1,\dots, i_k\}\in E(H)$ if and only if $j\in V_2$ and $i_1,\dots, i_k \in V_1$ are connected to $j$ in $G$. By the definition of $G$, vertices in $V_2$ have different sets of neighbors, hence the corresponding hypergraph $H$ has no multiple hyperedges.
It's easy to check that $H$ is a $(d,k)$-regular hypergraph on $n$ vertices.

Conversely, for any simple $(d,k)$-regular hypergraph $H\in \mathcal H(n,d,k)$, $X$ corresponds  the incidence matrix of  $H$, and we can associate to $H$ a $(n, nd/k,d,k)$-bipartite biregular graph $G$ whose adjacency matrix is $\begin{pmatrix}
	0 & X\\
	X^{\top}& 0
\end{pmatrix}$, and it has no two vertices in $V_2$ sharing the same set of neighbors.
\end{proof}

 From Lemma \ref{bijection}, the uniform distribution on $\mathcal G' \left(n,nd/k,d,k\right)$ for bipartite biregular graphs induces the uniform distribution  on $\mathcal H (n, d,k)$ for regular hypergraphs. With this observation, we are able to translate the results for spectra of random bipartite biregular graphs into results for spectra of random regular hypergraphs. Our first step is the following spectral gap result.

\begin{theorem}\label{matchinglower}
	Let $A$ be the adjacency matrix of a random $(d,k)$-regular hypergraph sampled uniformly from $ \mathcal H(n,d,k)$ with $d\geq k\geq 3$, then any eigenvalue $\lambda(A)\not=d(k-1)$  satisfies
$$|\lambda (A)- k+2| \leq 2\sqrt{(k-1)(d-1)}+\epsilon_n,$$
asymptotically almost surely with $\epsilon_n\to 0$ as $n\to\infty$.
\end{theorem}

\begin{remark}
For $k=2$, Theorem \ref{matchinglower} reduces to  Alon's second eigenvalue conjecture proved in \cite{friedman2003proof,bordenave2015new}. In terms of Ramanujan hypergraphs defined in \eqref{LISOLE}, the theorem implies almost every $(d,k)$-regular hypergraph is almost Ramanujan.	
\end{remark}

We start with the following lemma connecting the adjacency matrix of a regular hypergraph and its associated bipartite biregular graph.
\begin{lemma} \label{lem:matrixcorre}
Let $H$ be a $(d,k)$-regular hypergraph, and let $G$ be the corresponding bipartite biregular graph associated to $H$. Let $A_H$ be the adjacency matrix of $H$, and $A_G$ be the adjacency matrix of $G$ with the form
\begin{align}\label{MatrixB}
A_G=\begin{pmatrix}
	0 & X\\
	X^{\top}& 0
\end{pmatrix}.
\end{align} Then
	$XX^{\top}=A_H+dI$.
\end{lemma}

\begin{proof}
Let $V$ and $E$ be the vertex and hyperedge set of $H$ respectively. For $i\not=j$, we have
	\begin{align*}
	(XX^{\top})_{ij}=\sum_{e\in E} X_{ie}X_{je}
	     	   =\sum_{e\in E}\mathbf{1}_{\{i,j\}\in e}
	     	   =(A_{H})_{ij}.
	\end{align*}
For the diagonal elements, we have
$
	(XX^{\top})_{ii}=\sum_{e\in E} X_{ie}X_{ie}
	     	   =\deg(i)
	     	   =d.
$
Therefore $A_H+dI=XX^{\top}$.
\end{proof}

It's not hard to show that for $d\geq k$, all eigenvalues of $A_G$ from \eqref{MatrixB} occur in pairs $(\lambda,-\lambda)$, where $|\lambda|$ is a singular value of $X$, along with extra $(dn/k-n)$ many zero eigenvalues. 
The next result for random bipartite biregular graphs is given in \cite{brito2018spectral}. 
\begin{lemma}[Theorem 4 in \cite{brito2018spectral}]\label{lem:bipartitegap}
Let $A_G$ be the adjacency matrix of a random bipartite biregular  graph $G$ sampled uniformly from $\mathcal G(n,m,d_1,d_2)$, where $d_1\geq d_2$ are independent of $n$. Then:
\begin{enumerate}
	\item Its second eigenvalue $\lambda_2$ satisfies
	\begin{align}\label{lambda2}
	\lambda_2\leq \sqrt{d_1-1}+\sqrt{d_2-1}+o(1)
	\end{align} asymptotically almost surely as $n\to\infty$.
	\item Its smallest positive eigenvalue $\lambda_{\min}^+$ satisfies
	\begin{align}\label{lambdamin}
		\lambda_{\min}^+\geq \sqrt{d_1-1}-\sqrt{d_2-1}-o(1).
	\end{align}
\end{enumerate} 
\end{lemma}

We will use a result from \cite{mckay2010subgraphs} that estimates the probability that a random bipartite biregular graph sampled from $\mathcal G(n,m,d_1,d_2)$ contains some subgraph $L\subset K_{n,m}$, where $K_{n,m}$ is the complete bipartite graphs with $|V_1|=n, |V_2|=m$. Let $|L|$ be the number of edges of $L$ and we use the notation $[x]_a$ denotes the falling factorial $x(x-1)\cdots (x-a+1)$. For any vertex $v\in K_{n,m}$, let $g_v$ and $l_v$ denote the degree of $v$ considered as a vertex of $G$ and $L$ respectively. Let $l_{\max}$ be the largest value of $l_i$.

\begin{lemma}
[Theorem 3.5 in \cite{mckay2010subgraphs}] \label{lem:forbiddensubgraph}
Let $L\subset K_{n,m}$. If $|L|+2d_1(d_1+l_{\max}-2)\leq nd_1-1$, then 
\begin{align*}
\mathbb P\left( L\subset G\right)\leq \frac{\prod [g_i]_{l_i}}{[nd_1-4d_1^2-1]_{|L|}}.
\end{align*}
\end{lemma}

With Lemma \ref{lem:forbiddensubgraph}, we are able to estimate the probability that a random bipartite biregular graph sampled uniformly from $\mathcal G\left(n,nd/k,d,k\right)$ belongs to $\mathcal G'\left(n,nd/k,d,k\right)$.

\begin{lemma}\label{lem:contiguity}
Let $G$ be a bipartite biregular graph sampled uniformly from $\mathcal G\left(n,nd/k,d,k\right)$ such that $3\leq  k\leq d\leq \frac{n}{32}$. Then   
\[ \mathbb P\left(G\in \mathcal G'\left(n,nd/k,d,k\right) \right)=1-O\left( \frac{d^2}{nk^2}\right).\]
In particular, if $3\leq  k\leq d\leq \frac{n}{32}$ and $\frac{d}{k}=o(n^{1/2}),$ as $n\to\infty$,
\[\mathbb P\left(G\in \mathcal G'\left(n,nd/k,d,k\right) \right)\to 1.\] 
\end{lemma}

\begin{proof}
Let $V=V_1\cup V_2$ be the vertex set of a graph $G$  sampled uniformly from $\mathcal G\left(n,nd/k,d,k\right)$. Assume there exist two vertices $v_1,v_2\in V_2$ such that $v_1,v_2$ have the same neighborhood in $V_1$ denoted by $N(v_1,v_2)$. Since $\deg(v_1)=\deg(v_2)=k$,  $N(v_1,v_2)$ is of size $k$. Let $L$ be the subgraph induced by $N(v_1,v_2)$ and  $v_1,v_2$ (see Figure \ref{fig:forbidden}). Then  $|L|=2k$ and $l_{\max}=k$.  When $1\leq d\leq \frac{n}{32}$, the assumption in Lemma \ref{lem:forbiddensubgraph} holds. By Lemma \ref{lem:forbiddensubgraph}, we have 
	\begin{align*}
	\mathbb P(L\subset G)\leq \frac{(d(d-1))^k ([k]_k)^2}{[nd-4d^2-1]_{2k}}.	
	\end{align*}
The number of all possible vertex pairs in $V_2$ is $\binom{nd/k} 2$ and the number of all possible $k$ many distinct vertices in $V_1$ is $\binom{n}{k}$. Therefore for sufficiently large $n$, the probability that there exists two vertices in $V_2$ having the same neighborhood is at most
\begin{align*}
{\binom{nd/k}{2}}\cdot { \binom {n} {k}}\cdot  \frac{d^{2k} k^{2k}}{[nd-4d^2-1]_{2k}}&\leq \frac{(nd)^2}{2k^2}\left(\frac{ne}{k}\right)^k \left(\frac{dk}{nd-4d^2-2k}\right)^{2k}\\
&\leq \frac{(nd)^2}{2k^2}\left(\frac{ne}{k}\right)^k \left(\frac{2k}{n}\right)^{2k}=\frac{(nd)^2}{2k^2}\left(\frac{4ek}{n}\right)^k.
\end{align*}
Since $x\ln(x)$ is decreasing on $x\in (0,e^{-1})$ and $3\leq k\leq d\leq \frac{n}{32}$,  we have for large $n$, $\frac{12e}{n}\leq \frac{4ek}{n}<e^{-1}$ and 
$k\ln(4ek/n)\leq 3\ln(12e/n).$ 
Then  $\left(\frac{4ek}{n}\right)^k\leq \left(\frac{12e}{n}\right)^3.$
Therefore
\begin{align*}
\mathbb P\left(G\not\in \mathcal G'\left(n,nd/k,d,k\right)	\right)=O\left( \frac{d^2}{nk^2}\right). 
\end{align*}This completes the proof. 
\end{proof}

With the four lemmas above, we are ready to prove Theorem \ref{matchinglower}.

\begin{proof}[Proof of Theorem \ref{matchinglower}]
Let $A_H$ be the adjacency matrix of a random $(d,k)$-regular hypergraph with $d\geq k$. Then its associated bipartite biregular graph has adjacency matrix 
\eqref{MatrixB}, where $X$ is a $n\times nd/k$ matrix and   $XX^{\top}=A_H+dI$. 
Let $G$ be a bipartite biregular graphs chosen uniformly from  $\mathcal G\left(n,nd/k,d,k\right)$. From Lemma \ref{lem:contiguity}, we have 
\begin{align}
&\mathbb P\left(\lambda_2(A_G)\leq \sqrt{d-1}+\sqrt{k-1}+\epsilon_n \right) \notag\\
=&	\mathbb P\left(\lambda_2(A_G)\leq \sqrt{d-1}+\sqrt{k-1}+\epsilon_n \mid G\in \mathcal G'\left(n,nd/k,d,k\right)\right) \notag\\
 &\cdot \mathbb P\left(G\in \mathcal G'\left(n,nd/k,d,k\right)\right)+o(1).\label{eq:conditionprob} 
\end{align}
By Lemma \ref{lem:bipartitegap}(1), asymptotically almost surely $\lambda_2(A_G)\leq \sqrt{d-1}+\sqrt{k-1}+\epsilon_n$ for some sequence $\epsilon_n\to 0$. Therefore by \eqref{eq:conditionprob}, we have 
\begin{align*}
\lim_{n\to\infty}\mathbb P\left(\lambda_2(A_G)\leq \sqrt{d-1}+\sqrt{k-1}+\epsilon_n \mid G\in \mathcal G'\left(n,nd/k,d,k\right)\right)= 1.	
\end{align*}
The uniform measure on $\mathcal G(n,nd/k,d,k)$  conditioned on the event $\{G\in \mathcal G'\left(n,nd/k,d,k\right)\}$ is a uniform measure on $\mathcal G'\left(n,nd/k,d,k\right)$. Hence asymptotically almost surely a  bipartite biregular graph $G$ sampled uniformly from  $\mathcal G'\left(n,nd/k,d,k\right)$ satisfies \eqref{lambda2}.

 Note that $G$ also satisfies \eqref{lambdamin} asymptotically almost surely. 
Since there is a bijection between  $\mathcal G'\left(n,nd/k,d,k\right)$ and $\mathcal H(n,d,k)$ described in Lemma \ref{bijection}, 
by \eqref{lambda2} and  Lemma \ref{lem:matrixcorre}, we have with high probability,
$\lambda_2(XX^{\top})=\lambda_2^2(A_G)\leq d+k-2+2\sqrt{(k-1)(d-1)}+o(1).
$
And it implies with high probability,
	\begin{align}\label{lower}
	\lambda_2(A_H)-k+2\leq 2\sqrt{(k-1)(d-1)}+o(1).
	\end{align}
Similarly, from \eqref{lambdamin}, 
for the smallest eigenvalue $\lambda_n(A_H)$, we have with high probability,
$$\lambda_n(A_H)+d=\lambda_n(XX^{\top})={\lambda_{\min}^+}(A_G)^2\geq d+k-2-2\sqrt{(d-1)(k-1)}-o(1),$$
which implies with high probability,
\begin{align}\label{upper}
\lambda_n(A_H)-k+2\geq -2\sqrt{(d-1)(k-1)}-o(1).	
\end{align}
Combining \eqref{lower} with \eqref{upper}, and note that the largest eigenvalue of $A$ is $d(k-1)$, we have 
$
	|\lambda-k+2|\leq 2\sqrt{(d-1)(k-1)}+o(1)
$
for any eigenvalue $\lambda\not=d(k-1)$ asymptotically almost surely. This completes the proof of Theorem \ref{matchinglower}.
\end{proof}

\section{Spectra of the non-backtracking operators}\label{sec:NBO}

Following the definition in \cite{angelini2015spectral}, for a hypergraph $H=(V,E)$, its \textit{non-backtracking operator} $B$ is a square matrix indexed by oriented hyperedges $\vec{E}=\{(i,e): i\in V, e\in E, i\in e\}$ with entries given by 
$$B_{(i, e), (j,f)}=\begin{cases}
	1 & \text{if $j\in e\setminus\{i\}, f\not=e$,}\\
	0 & \text{otherwise,}
\end{cases}
$$
for any oriented hyperedges $(i, e),(j, f)$. This is a generalization of the graph non-backtracking operators to hypergraphs. In \cite{angelini2015spectral} a spectral algorithm was proposed for solving community detection problems on sparse random hypergraph, and it uses the  eigenvectors of the non-backtracking operator defined above. To obtain theoretical guarantees for this spectral algorithm, we need to prove a spectral gap for the non-backtracking operator.  To the best of our knowledge, this operator has not been rigorously analyzed for any random hypergraph models. In the first step, we study the spectrum of the non-backtracking operator for the random regular hypergraphs. From the bijection in Lemma \ref{bijection}, it is important to find its connection to the non-backtracking operator of the corresponding bipartite biregular graph.

Consider a bipartite graph $G=(V(G),E(G))$ with $V(G)=V_1(G)\cup V_2(G)$. The non-backtracking operator $B_G$ of $G$ is a matrix indexed by the set of oriented edges 
$\vec{E}(G)=\{e=(i,j): i,j\in V(G), e\in E(G)\}$ with dimension $2|E(G)|\times 2|E(G)|$. For  an oriented edge $e=(i,j)$ and  $f=(s,t)$,  define  $B_G$ as 
\begin{align*}
	(B_G)_{ef}=\begin{cases}
1, & \text{if $j=s$ and $t\not=i$;}\\
0, & \text{otherwise.}	
\end{cases}
\end{align*}

We order the elements of $\vec{E}$ as $\{e_1,\dots, e_{2|E(G)|}\}$, so that the first $|E(G)|$ oriented edges have starting vertices from $V_1$ and ending vertices in $V_2$. In this way, we can write
\begin{align}\label{eq:BGG}
B_G=\left(\begin{matrix}
	0 & M\\
	N& 0
\end{matrix}	\right),
\end{align}
where $M,N$ are $|E|\times |E|$ matrices with entries in $\{0,1\}$. The following lemma connects the non-backtracking operator $B_H$ of a hypergraph $H$ to the non-backtracking operator $B_G$ of its associated bipartite graph $G$.

\begin{lemma}\label{BMN}
	Let $B_H$ be a non-backtracking operator of $H$. Let $G$ be its associated bipartite graph with a non-backtracking operator  given by \eqref{eq:BGG}. Then $B_H=MN$.
\end{lemma}

\begin{proof}
Since $
B_G^2=\left(\begin{matrix}
	MN & 0\\
	 0  & NM
\end{matrix}	\right),
$ it suffices to show the $|E|\times |E|$ submatrix $MN$ in $B^2_G$ is $B_H$. From our construction of the associated bipartite graph, we know $V(G)=V_1\cup V_2$ and $V_1=V(H), V_2=E(H)$. The oriented edges with starting vertices from $V_1$ and ending vertices from $V_2$ can be denoted by $(i, e)$, where $i\in V(H), e\in E(H)$. Then for any $(i, e), (j,f)$ in $\vec{E}(G)$, we have 
\begin{align*}
(B_G^2)_{(i,e),(j, f)}&=\sum_{(k,g)\in \vec{E}(G)} (B_G)_{(i, e), (k, g)}	(B_G)_{(k, g), (j,f)}=\sum_{(k, g)\in \vec{E}(G)}\mathbf{1}_{\{e=k,g\not=i\}}\mathbf{1}_{\{j=g,f\not=k\}}\\
&=\mathbf{1}_{\{(e,j)\in \vec{E}(G)\}}\mathbf{1}_{\{j\not=i, f\not=e \}}=\mathbf{1}_{\{j\in e, j\not=i, f\not=e\}}=(B_H)_{(i,e),(j, f)}.
\end{align*}
	Hence $B_H=MN$, this completes the proof.
\end{proof}

\begin{remark}
	Lemma \ref{BMN} is true for any hypergraphs, including non-uniform hypergraphs.
\end{remark}

If $H$ is a $(d,k)$-regular hypergraph, then
$G$ is a $(d,k)$-bipartite biregular graph with  $|V_1(G)|=n, |V_2(G)|=nd/k$. Our next lemma for the spectrum of $B_G$ is from from \cite{brito2018spectral}. 

\begin{lemma}[Lemma 2 in \cite{brito2018spectral}]\label{BG} Let $G$ be a $(d,k)$-bipartite biregular graph with $n$ vertices.
	Any eigenvalue of $B_G$ belongs to one of the following  categories:
	\begin{enumerate}
		\item $\pm 1$ are both eigenvalues with multiplicity $|E(G)|-|V(G)|=n(d-1)-nd/k$.
		\item $\pm i\sqrt{d-1}$ are eigenvalues with multiplicity $n-r$, where $r$ is the rank of $X$.
		\item  $\pm i\sqrt{k-1}$ are eigenvalues with multiplicity $nd/k-r$.
		\item Every pair of non-zero eigenvalues $(-\xi, \xi)$ of the adjacency matrix $A_G$ generates  exactly $4$ eigenvalues of $B_G$ with the equation $ \lambda^4-(\xi^2-d-k+2)\lambda^2+(k-1)(d-1)=0 .$	\end{enumerate}
\end{lemma}

We have the following characterization of eigenvalues for $B_H$ of a $(d,k)$-regular hypergraph $H$. It follows immediately from Lemma \ref{BMN} and Lemma \ref{BG}.

\begin{theorem}\label{BH}
Let $H$ be a $(d,k)$-regular hypergraph on $n$ vertices and $G$ be its associated $(d,k)$-bipartite biregular graph with adjacency matrix $A_G$ given in \eqref{MatrixB}. All eigenvalues of $B_H$ can be classified into the following:
\begin{enumerate}
	\item $1$ with multiplicity $n(d-1)-nd/k$.
	\item $-(d-1)$ with multiplicity $n-r$, where $r$ is the rank of $X$.
	\item $-(k-1)$ with multiplicity $nd/k-r$.
	\item Every pair of non-zero eigenvalues $(-\xi,\xi)$ of $A_G$ generates exactly $2$ eigenvalues of $B_H$ with the equation:
	$\lambda^2-(\xi^2-d-k-2)\lambda+(k-1)(d-1)=0.
	$ 
\end{enumerate}	
\end{theorem}

Let $G$ be an associated $(d,k)$-bipartite biregular graph of a regular hypergraph $H$. From \cite[Section 2]{brito2018spectral}, $\pm \sqrt{(d-1)(k-1)}$ are eigenvalues of $B_G$  with multiplicity $1$. Then from Theorem \ref{BH}, $B_H$ has an eigenvalue $\lambda_1(B_H)=(d-1)(k-1)$ with multiplicity $1$. 
From  \cite[Theorem 3]{brito2018spectral},
for random $(d,k)$-bipartite biregular graphs,  the second largest eigenvalue (in absolute value) $\lambda_2(B_G)$ satisfies
\begin{align}\label{eq:BG}
|\lambda_2(B_G)|\leq ((k-1)(d-1))^{1/4}+o(1)	
\end{align}
 asymptotically almost surely as $n\to\infty$. Therefore  from the discussion above, together with Lemma \ref{lem:contiguity}, we obtain the following spectral gap result for $B_H$.
\begin{theorem} Let $H$ be a random $(d,k)$-regular hypergraph sampled from $\mathcal H(n,d,k)$, with $d\geq k\geq 3$. Then any eigenvalue $\lambda$ of $B_H$ with $\lambda\not=(d-1)(k-1)$ satisfies 
\[|\lambda|\leq ((k-1)(d-1))^{1/2}+o(1)\] asymptotically almost surely as $n\to\infty$.
\end{theorem}

\section{Empirical spectral distributions}\label{sec:ESD}
In the last section, we study the empirical spectral distribution of the adjacency matrix of a random regular hypergraph.
We define the \textit{empirical spectral distribution} (ESD) of a symmetric $n\times n$ matrix $M$ to be the probability measure $\mu_n$ on $\mathbb R$ given by
\[\mu_n=\frac{1}{n}\sum_{i=1}\delta_{\lambda_i},\]
 where $\delta_x$ is the point mass at $x$ and $\lambda_1,\dots, \lambda_n$ are the eigenvalues of $M$.  We always assume $d\geq k$ (see Remark \ref{dualremark}).   Feng and Li in \cite{feng1996spectra} derived the limiting  ESD for a sequence of connected $(d,k)$-regular hypergraphs with fixed $d,k$ as follows. The  definition of \textit{primitive cycles} in \cite{feng1996spectra} is the same as \textit{cycles} in our Definition \ref{cycledefinition}.

\begin{theorem}[Theorem 4 in \cite{feng1996spectra}]\label{feng}
	Let $H_n$ be a family of connected $(d,k)$-regular hypergraphs with $n$ vertices. Assume for each integer $l\geq 1$, 
	the number of cycles of length $l$ is $o(n)$. Denote $q=(d-1)(k-1)$. For fixed $d\geq k\geq 3$, the empirical spectral distribution of $ M_n:=\frac{A-(k-2)}{\sqrt{(d-1)(k-1)}}$ converges weakly in probability to the measure $\mu$ supported on $[-2,2]$, whose density function is given by
	\begin{align}\label{LSD}
	f(x):=\frac{1+\frac{k-1}{q}}{(1+\frac{1}{q}-\frac{x}{\sqrt{q}})(1+\frac{(k-1)^2}{q}+\frac{(k-1)x}{\sqrt{q}})}\cdot \frac{1}{\pi}\sqrt{1-\frac{x^2}{4}} dx.
	\end{align}
	\end{theorem}

We prove that for uniform random regular hypergraphs, the assumptions in Theorem \ref{feng} hold with high probability, which implies the convergence of ESD  in probability for random regular hypergraphs.  

\begin{lemma} \label{lem:connect}
Let $H$ be a random $(d,k)$-regular hypergraph with fixed $d\geq k\geq 3$. Then $H$ is connected asymptotically almost surely.
\end{lemma}

\begin{proof}
	$H$ is connected if and only if its associated bipartite biregular graph $G$ is connected. The first eigenvalue for the $(d,k)$-bipartite biregular graph $G$ is $\lambda_1=\sqrt{dk}$ and we know from Lemma \ref{lem:bipartitegap} and Lemma \ref{lem:contiguity}, for a uniformly chosen  random regular hypergraph $H$, the corresponding bipartite biregular graph $G$ satisfies  $\lambda_2\leq \sqrt{d-1}+\sqrt{k-1}+o(1)$ 
	asymptotically almost surely. 	Note that for $d,k\geq 2,$
	$\sqrt{d-1}+\sqrt{k-1}=\sqrt{dk}$ if and only if $d=k=2$. So when $d\geq k\geq  3$, for sufficiently large $n$, the first eigenvalue has multiplicity one with high probability.
	If $G$ is not connected, we can decompose $G$ as $G=G_1\cup G_2$ such that there is no edge between $G_1$ and $G_2$. Then $G_1, G_2$ are  both bipartite biregular graphs with the largest eigenvalue $\sqrt{dk}$. However, that implies $G$ satisfies $\lambda_2=\sqrt{dk}$, a contradiction.
\end{proof}

 The following lemma shows the number of  cycles of length $l$ in $H$ is $o(n)$ asymptotically almost surely.

\begin{lemma}\label{lem:cycle}
	 Let $H_n$ be a random $(d,k)$-regular hypergraph.  For each integer $l\geq 1$, the number of cycles of length $l$ in $H_n$ is $o(n)$ asymptotically almost surely.
\end{lemma}
\begin{proof} By Lemma \ref{lem:contiguity},
it is equivalent to show the number of cycles of length $2l$ for a random bipartite biregular graph, denoted by $X_l$, is $o(n)$ with high probability.   
 From  \cite[Proposition 4 ]{dumitriu2016marvcenko}, when $d,k,l$ are fixed, $
\mathbb E[X_l]=O(1), $ and $\textnormal{Var}[X_l]=O(1).
$
 Then by Chebyshev's inequality, 
 $$\mathbb P\left(|X_l-\mathbb EX_l|\geq \frac{n}{\log n}\right)=O\left(\frac{\log^2(n)}{n^2}\right).
$$ Hence $X_l=o(n)$ asymptotically almost surely.
\end{proof}

Combining Theorem \ref{feng}, Lemma \ref{lem:connect} and Lemma \ref{lem:cycle}, we have the following theorem for the ESDs of random regular hypergraphs with fixed $d,k$:

\begin{theorem}
	Let $A_n$ be the adjacency matrix of a random $(d,k)$-regular hypergraph on $n$ vertices.  Let $ M_n:=\frac{A-(k-2)}{\sqrt{(d-1)(k-1)}}.$ For fixed $d\geq k\geq 3$,  the empirical spectral distribution of $M_n$ converges in probability to a measure $\mu$ with density function $f(x)$ given in \eqref{LSD}.
\end{theorem}

\begin{remark}
When $k=2$, $f(x)$ is the density of the Kesten-McKay law \cite{mckay1981expected} with a different scaling factor. For $k\geq 3$, the limiting distribution in \eqref{LSD} is  not symmetric (i.e.~$f(x)\not= f(-x)$), which is quite different from the random graph case.  For random bipartite biregular graphs with bounded degrees, the limit of the ESDs was derived in \cite{godsil1988walk}, and later in \cite{bordenave2010resolvent} using different methods.
\end{remark}

In \cite{feng1996spectra}, the cases where $d,k$ grow with $n$ have not been discussed. With the results on random bipartite biregular graphs from \cite{dumitriu2016marvcenko,Tran2019}, we can get the following result in this regime.
\begin{theorem}\label{global} Let $A_n$ be the adjacency matrix of a random $(d,k)$-regular hypergraph on $n$ vertices. For $d\to\infty$ with $\frac{d}{k}\to\alpha\geq 1$ and $d=o(n^{1/2})$, the empirical spectral distribution of $ M_n:=\frac{A_n-(k-2)}{\sqrt{(d-1)(k-1)}}$ converges in probability to  a measure supported on $[-2,2]$ with a density function 	\begin{align}\label{mua}
	g(x)=\frac{\alpha}{1+\alpha+\sqrt{\alpha}x}\frac{1}{\pi}\sqrt{1-\frac{x^2}{4}} .
	\end{align}
\end{theorem}

To prove Theorem \ref{global}, we will apply the following results for the global law of random bipartite biregular graphs.
\begin{theorem}[Theorem 1 in \cite{dumitriu2016marvcenko} and Corollary 2.2 in \cite{Tran2019}]\label{lem:BB}
	Let $A_G$ be the adjacency matrix of a random bipartite biregular graph sampled from $\mathcal G(n,m,d,k)$ with $n\leq m$, $\frac{d}{k}\to \alpha\geq 1$, and $d=o(n^{1/2})$ as $n\to\infty$. Then the ESD of 
$ \frac{A_G}{\sqrt{k}}$ converges asymptotically almost surely to a distribution supported on $[-b,-a]\cup [a,b]$ with density  
\begin{align}\label{bbrg}
	h(x):=\frac{\alpha}{(1+\alpha)\pi|x|}\sqrt{(b^2-x^2)(x^2-a^2)},
	\end{align} and a point mass of $\frac{\alpha-1}{\alpha+1}$ at $0$, where $a=1-\alpha^{-1/2}, b=1+\alpha^{-1/2}$.
\end{theorem}

\begin{proof}[Proof of Theorem \ref{global}] Let $A_G$ be the adjacency matrix of a random $(d,k)$-bipartite biregular graph sampled from $\mathcal G(n,nd/k,d,k)$.
	Since $X$ is a $n\times m$ matrix with $n\leq m$, the ESD of $ \frac{XX^{\top}}{k}$ is the distribution of the squares of the nonzero eigenvalues of $ \frac{A_G}{\sqrt{k}}$, and from \eqref{bbrg}, the ESD of $ \frac{XX^{\top}}{k}$  is supported on $[a^2,b^2]$ with the density function given by
\begin{align}\label{muX}
\tilde{h}(x)=\frac{\alpha}{2\pi x}\sqrt{(b^2-x)(x-a^2)}
\end{align}
asymptotically almost surely. By Lemma \ref{lem:contiguity}, the same statement holds for a random bipartite biregular graph $G$ sampled uniformly from $\mathcal G'\left(n,nd/k,d,k\right)$.
Since the adjacency matrix of the corresponding regular hypergraph $H$ is $A_n=XX^{\top}-d$, by scaling, this implies that the ESD of $M_n= \frac{A_n-k-2}{\sqrt{(d-1)(k-1)}}$ is supported on $[-2,2]$ and the density is  given by \eqref{mua}.
\end{proof}

The convergence of empirical spectral distributions on short intervals (also known as the local law) for random bipartite biregular graphs was studied in \cite{dumitriu2016marvcenko,Tran2019,yang2017local}. Universality of eigenvalue statistics was studied in \cite{yang2017bulk}. All of these local eigenvalue statistics   can be translated to random regular hypergraphs via the bijection in Lemma \ref{bijection}. As an example, we translate the following result about the local law for random bipartite biregular graphs in  \cite{dumitriu2016marvcenko} to random regular hypergraphs. 

\begin{theorem}\label{localreg}
	Let $H$ be a random $(d,k)$-regular hypergraph on $n$ vertices satisfying $d\to\infty$ as $n\to\infty$, $ \frac{d}{k}\to\alpha\geq 1$
	and $ \log k =o\left(\sqrt{\log n}\right)$. Let $A$ be the adjacency matrix of $H$ and $\mu_n$ be the ESD of $ M:=\frac{A-d}{\sqrt{(d-1)(k-1)}}$ and $\mu$ be the limiting ESD defined in \eqref{mua}. For any $\epsilon>0$, there exists a constant $C_{\epsilon}$ such that for all sufficiently large $n$ and $\delta>0$, for any interval $I\subset [-2+\epsilon,2]$  with length  $ |I|\geq \frac{4(1+\sqrt{\alpha})^2}{\sqrt\alpha} \max\left\{2\eta, \frac{\eta}{-\delta\log\delta} \right\}$, it holds that
$|\mu_n(I)-\mu(I)|\leq \delta C_{\epsilon}|I|
$
with probability $1-o(1)$, where $\eta$ is given by the following quantities:
\begin{align}\label{eq:defeta}
h&=\min\left\{\frac{\log n}{9(\log k)^2}, k\right\},\quad 	
r=e^{1/h},\quad \eta=r^{1/2}-r^{-1/2}.
\end{align}
\end{theorem}

We prove Theorem \ref{localreg} from 
the following   local law for random bipartite biregular graphs in \cite{dumitriu2016marvcenko}.
\begin{lemma}[Theorem 3 in \cite{dumitriu2016marvcenko}]\label{DT}
	Let $G$ be a random $(d,k)$-bipartite biregular graph on $n+nd/k$ vertices satisfying $d\to\infty$ as $n\to\infty$ and 
	$
	 \log k=o\left(\sqrt{\log n}\right),
	 \frac{d}{k}\to\alpha\geq 1.
	$
Let $A_G$ be the adjacency matrix of $G$ and $\mu_n$ be the ESD of $\frac{A_G}{\sqrt{k-1}}$ and let $\mu$ be the measure defined in \eqref{bbrg}. For any $\epsilon>0$, there exists a constant $C_{\epsilon}$ such that for all sufficiently large $n$ and $0<\delta<1$, for any interval $I\subset \mathbb R$ avoiding $[-\epsilon,\epsilon]$ and with length $ |I|\geq \max\left\{2\eta, \frac{\eta}{-\delta\log\delta} \right\}$, 
\begin{align}\label{eq:locallaw}
|\mu_n(I)-\mu(I)|\leq \delta C_{\epsilon}|I|
\end{align}
with probability $1-o(1/n)$, where $\eta$ is given in \eqref{eq:defeta}.
\end{lemma}

\begin{proof}[Proof of Theorem \ref{localreg}] For any interval $I\subset \mathbb R$ and a symmetric matrix $M$, we denote $N_I^M$ to be the number of eigenvalues of $M$ in the interval $I$. For a random $(d,k)$-regular hypergraph $H$ with adjacency matrix $A$, let $G$ be its associated bipartite biregular graph and $A_G$ be the adjacency matrix of $G$. With Lemma \ref{lem:contiguity}, we know \eqref{eq:locallaw} holds for $A_G$ with probability $1-o(1)$. Recall that the ESD of $ W:=\frac{XX^{\top}}{k-1}$ is the distribution of the squares of the nontrivial eigenvalues of $ \frac{A_G}{\sqrt{k-1}}$. Let $ M:=\frac{A_G-k-2}{\sqrt{(d-1)(k-1)}}$. Consider any interval $ I_1=[\beta,\gamma]\subset [-2+\epsilon,2]$ with length 
\begin{align}\label{eq:I1}
 |I_1|\geq \frac{4(1+\sqrt{\alpha})^2}{\sqrt{\alpha}}\max\left\{2\eta, \frac{\eta}{-\delta\log\delta} \right\}.	
\end{align}
Let 
$I_2=\left[\sqrt{\frac{d-1}{k-1}}\beta+\frac{d+k+2}{k-1},\sqrt{\frac{d-1}{k-1}}\gamma+\frac{d+k+2}{k-1}\right]:=[\beta',\gamma']$ be a shifted and rescaled interval from $I_1$. 
We have 
\begin{align}
|I_2|&=\gamma'-\beta'=(\sqrt{\alpha}+o(1))(\beta-\gamma)=(\sqrt{\alpha}+o(1))|I_1|,\label{eq:para1}\\
	\beta' &=\sqrt{\alpha}\beta+\alpha+1+o(1)=(\sqrt{\alpha}-1)^2+\sqrt{\alpha}\epsilon+o(1)>(\sqrt{\alpha}-1)^2+\frac{\epsilon}{2}, \label{eq:para2}\\
	\gamma' &\leq 2\sqrt{\frac{d-1}{k-1}}+\frac{d+k+2}{k-1}=2\sqrt{\alpha}+\alpha+1+o(1)=(1+\sqrt{\alpha})^2+o(1) .\label{eq:para3}
\end{align}
Let $I_3:=[\sqrt{\beta'},\sqrt{\gamma'}]$. From the eigenvalue relation between $M,W$ and $A_G$, we have 
$N_{I_1}^M=N_{I_2}^W=2N_{I_3}^{\frac{A_G}{\sqrt{k-1}}}.$
Note that from \eqref{eq:para1} and \eqref{eq:para3}, the interval length of $I_3$ satisfies
\begin{align*}
|I_3|=\sqrt{\gamma'}-\sqrt{\beta'}=\frac{\gamma'-\beta'}{\sqrt{\gamma'}+\sqrt{\beta'}}\geq\frac{(\sqrt{\alpha}+o(1))|I_1|}{2(1+\sqrt{\alpha})^2+o(1)}\geq 	\max\left\{2\eta, \frac{\eta}{-\delta\log\delta} \right\},
\end{align*}
where the last inequality is from \eqref{eq:I1}. From \eqref{eq:para2},
\begin{align}\label{eq:I3}
	|I_3|=\frac{\gamma'-\beta'}{\sqrt{\gamma'}+\sqrt{\beta'}}\leq \frac{\gamma'-\beta'}{2\sqrt{\beta'}}\leq \frac{\sqrt{\alpha}+o(1)}{2(\sqrt{\alpha}-1)^2+\epsilon}|I_1|.
\end{align}
From Lemma \ref{DT}, since $\sqrt{\beta'}\geq \sqrt{\epsilon/2}$, $I_3$ is an interval avoiding $[-\sqrt{\epsilon/2},\sqrt{\epsilon/2}]$, hence  there exists a constant $C_{\epsilon}$ such that 
\begin{align}\label{small}
\left|\frac{1}{n+\frac{nd}{k}}N_{I_3}^{\frac{A_G}{\sqrt{k-1}}}-\mu_G(I_3)\right|\leq \delta C_{\epsilon}|I_3|,
\end{align}
where $\mu_G$ is the limiting measure defined in \eqref{bbrg}. Let $\mu_X$ be the limiting measure defined in \eqref{muX} and $\mu_A$ be the limiting measure defined in \eqref{mua}.
Note that $\mu_A(I_1)=\mu_X(I_2)=2(\alpha+1)\mu_G(I_3).$ Therefore \eqref{small} implies
\begin{align}\label{eq:NIM}
	\left|\frac{1}{2(n+\frac{nd}{k})}N_{I_1}^M-\frac{1}{2(\alpha+1)}\mu_A(I_1)\right| &\leq \delta C_{\epsilon}|I_3|.
	\end{align}
	Let $\mu_{n}$ be the ESD of $M$. From \eqref{eq:NIM}, we get	
\begin{align}
	&\left|\mu_{n}(I_1)-\mu_A(I_1)\right|=\left|\frac{1}{n} N_{I_1}^M -\mu_A(I_1)\right|= 2\left(1+\frac{d}{k}\right)\left|\frac{1}{2(n+\frac{nd}{k})}N_{I_1}^M-\frac{1}{2(1+\frac{d}{k})}\mu_A(I_1)\right|\notag\\
	\leq& 2\left(1+\frac{d}{k}\right)\left(\left|\frac{1}{2(n+\frac{nd}{k})}N_{I_1}^M-\frac{1}{2(1+\alpha)}\mu_A(I_1)\right| +\left|\frac{1}{2(\alpha+1)}-\frac{1}{2(1+\frac{d}{k})} \right|\mu_A(I_1)\right)\notag\\
	\leq & 2\left(1+\frac{d}{k}\right)\delta C_{\epsilon}|I_3|+\frac{1}{\alpha+1}\left|\frac{d}{k}-\alpha\right|\mu_A(I_1)\notag\\
	\leq & \delta\frac{(2+2\alpha)\sqrt{\alpha}+o(1)}{2(\sqrt{\alpha}-1)^2+\epsilon}C_{\epsilon}|I_1|+o(\mu_A(I_1))	\leq \frac{4\delta (1+\alpha)\sqrt{\alpha}}{2(\sqrt{\alpha}-1)^2+\epsilon}C_{\epsilon}|I_1|:=\delta C_{\epsilon}'|I_1|, \label{eq:line2}
\end{align} 
	where  the first inequality in \eqref{eq:line2} is from \eqref{eq:I3}, and $ C_{\epsilon}'$ is a constant depending on $\alpha$ and $\epsilon$. This completes the proof of Theorem \ref{localreg}.
\end{proof}

\section*{Acknowledgments} We thank Sebastian Cioab\u{a}
 and Kameron Decker Harris for helpful comments. This work was partially supported by
NSF DMS-1949617.

 \bibliographystyle{plain}
\bibliography{ref.bib}

\end{document}